\documentclass[12pt]{amsart}
\usepackage{mystyle}
\title[Conformal Gauss-Bonnet for Four-Manifolds with Corner]{Extrinsic Curvature and Conformal Gauss-Bonnet for Four-Manifolds with Corner}
\author{Stephen E. McKeown}
\address{Department of Mathematical Sciences, FO 35; University of Texas at Dallas; 800 W. Campbell Road; Richardson, TX 75080; USA}
\email{stephen.mckeown@utdallas.edu}
\thanks{Research partially supported by NSF RTG Grant DMS-1502525.}
\subjclass[2010]{Primary 53C40, 53C18; Secondary 58J99.}
\date{\today}

\begin{document}

\maketitle

\begin{abstract}
	This paper defines two new extrinsic curvature quantities on the corner of a four-dimensional Riemannian manifold with corner. One of these is a pointwise conformal invariant, and the conformal
	transformation of the other is governed by a new linear second-order pointwise conformally invariant partial differential operator. The Gauss-Bonnet theorem is then stated in terms of these quantities.
\end{abstract}

\section{Introduction}

The extension of the Gauss-Bonnet theorem to manifolds of higher even dimension than two was one of the great achievements of mid-twentieth century geometry. (See \cite{f40,a42,aw43,c44,c45} and
the classic exposition \cite[Volume V]{spi99}.) Nevertheless, the Pfaffian is a distinctly
more complicated object than the Gaussian curvature, and mining the theorem for its geometric consequences has long been a challenging task.

One successful application in dimension four has been the use of the theorem to identify and analyze conformally invariant differential operators and associated curvature quantities. It was first observed in
\cite{bg94} that the Gauss-Bonnet formula on a closed Riemannian four-manifold $(X,g)$ could be rewritten as
\begin{equation*}
	4\pi^2\chi(X) = \oint_X \left( \frac{1}{8}|W_g|_g^2 + \frac{1}{2}Q \right)dv_g,
\end{equation*}
where $W$ is the Weyl curvature tensor and $Q$ is the so-called $Q$-curvature (see \cite{bo91}). 
The first term in the integrand yields a pointwise conformal invariant, so this furnishes a proof that $\oint_X Q dv_g$ is a global conformal
invariant. One can also calculate that, under the conformal change $\tilde{g} = e^{2\omega}g$, the $Q$-curvature transforms pointwise by
\begin{equation}
	\label{QPeq}
	e^{4\omega}\widetilde{Q} = Q + P_4\omega,
\end{equation}
where $P_4$ is the Paneitz operator, a linear elliptic fourth-order PDO that satisfies the pointwise conformal invariance property $\widetilde{P}_4 = e^{-4\omega}P_4$. Thus,
this form of the Gauss-Bonnet formula offers a tractable link between conformal geometry and topology. The relationship between the Gauss-Bonnet formula
and $Q$-curvature motivated much important geometric analysis, including Alexakis's structure theorem on global conformal invariants (\cite{a12}), the conformally invariant sphere theorem \cite{cgy03},
and the topological result in \cite{gursk99}, among numerous others.

Motivated by the work of Branson and Gilkey, Chang and Qing in \cite{cq97i} showed that the Chern-Gauss-Bonnet formula for compact Riemannian
four-manifolds with boundary can also be written in a form that interacts
in a particularly nice way with conformal geometry, and in so doing, they identified important new boundary curvature quantities. To be specific, let
$(X^4,g)$ be such a manifold with boundary $M^3 = \partial X$ and induced metric $h = g|_{TM}$. The result of Chang and Qing was that the Gauss-Bonnet formula in this setting can be written
\begin{equation}
	\label{GBB}
	4\pi^2\chi(X) = \int_X \left( \frac{1}{8}|W_g|_g^2 + \frac{1}{2}Q \right)dv_g + \oint_M (\mathcal{L} + T) dv_h,
\end{equation}
where $\mathcal{L}$ is an extrinsic pointwise conformal invariant of weight $-3$ along the boundary, and $T$ is an extrinsic third-order curvature quantity along $M$ which, under conformal transformations,
transforms by
\begin{equation}
	\label{TPeq}
	e^{3\omega}\widetilde{T} = T + P_3\omega,
\end{equation}
where $P_3:C^{\infty}(X) \to C^{\infty}(M)$ is a linear third-order differential operator that is pointwise conformally invariant of weight $-3$. Both the curvature quantity $T$ and the formula (\ref{GBB}) have since
proved important in conformal geometry. For example, they were used in (\cite{cqy08}) to give an analysis of the renormalized volume on asymptotically hyperbolic Einstein manifolds; were used 
in \cite{n08} and \cite{n09} to study Escobar-Yamabe type problems
for $Q$-curvature on manifolds with boundary; yielded the important defect theorem for total $Q$-curvature in \cite{cqy00}; and have proven important in the analysis of compactness for conformally compact Einstein
metrics (see \cite{cgq18}). A Sobolev trace inequality was proved in \cite{ac18}. There have been numerous further applications besides. (See also \cite{j09}.)

In light of the importance of the $T/P_3$ pair, there has been interest in finding generalizations to other situations. For example, Case and Luo in \cite{cl19}, along with Gover and Peterson in \cite{gp18},
defined fifth-order conformally covariant boundary operators and associated fifth-order extrinsic boundary curvatures (as well as various lower orders) such that the operators parametrized conformal changes
in the curvatures, whose total boundary integral was conformally invariant and which, when paired with the sixth-order GJMS operator (the sixth-order generalization of the Paneitz operator) yield a symmetric pair. The
former paper also proves a Sobolev trace theorem.

In this paper, I extend the result of \cite{cq97i} in a different direction: I consider four-manifolds with corners of codimension two, and in particular, derive an order-two conformal curvature/operator pair along
the corner. I then show that the Gauss-Bonnet formula may be written on such a manifold in terms of these quantities and a pointwise conformal invariant.
The proof is based on the Allendoerfer-Weil formula of \cite{aw43}. One forthcoming application is discussed below; other expected applications are to extend theorems and arguments in the conformal geometry of
four-manifolds to situations with higher boundary codimension.

To state the results cleanly, we let $X$ be a compact four-manifold with corners, having two embedded boundary hypersurfaces
$M^3$ and $N^3$, whose common boundary $\Sigma^2 = M \cap N$ is a smooth closed surface. (Nothing prevents $X$ from having more than one corner component, or a mixture of closed and boundaried boundary hypersurfaces.
The modifications necessary in that case will be obvious, but it will be notationally much cleaner to consider this simpler situation.) We assume that $M$ and $N$
meet transversely, at an angle which may vary but which lies at each point between $0$ and $\pi$.

For any three-dimensional hypersurface, we let $h$ denote the induced metric, $\mu$ denote an inward-pointing unit vector, $L$ the second fundamental form computed with respect to this vector, and
$H = h^{\mu\nu}L_{\mu\nu}$ the mean curvature. For $M$ and $N$ respectively, we will decorate these various terms with a subscript to indicate the hypersurface intended; so, for example, we will have
$h_M, \mu_M, L_M$, and $H_M$.

The $Q$-curvature is defined on $X$ by
\begin{equation*}
	6Q = -\Delta_gR + R^2 - 3R^{ij}R_{ij},
\end{equation*}
where $R$ here denotes the scalar or Ricci curvatures of $g$ and the Laplacian is a negative operator. The $T$-curvature along a hypersurface, meanwhile, is defined by
\begin{align*}
	T &= -\frac{1}{12}\mu(R_g) - \mathring{L}^{\mu\nu}R_{\mu\nu}^g + \mathring{L}^{\mu\nu}R_{\mu\nu}^h - \frac{1}{2}H|\mathring{L}|_h^2 + \frac{2}{3}\mathring{L}^3\\
	&\quad+\frac{1}{6}HR_h - \frac{1}{27}H^3 - \frac{1}{3}\Delta_hH.
\end{align*}
The $\mathcal{L}$-curvature along a hypersurface is
\begin{equation*}
	\mathcal{L} = \mathring{L}^{\mu\nu}R_{\mu\nu}^g - 2\mathring{L}^{\mu\nu}R_{\mu\nu}^h + \frac{2}{3}H|\mathring{L}|_h^2 - \mathring{L}^3.
\end{equation*}
For ease of reference, we also record the formula
\begin{equation*}
		\begin{split}
			P_3^gu =& \frac{1}{2}\mu\Delta_gu - \Delta_{h}\mu(u) - H\Delta_hu - \mathring{L}^{\mu\nu}\nabla_{\mu}^h\nabla_{\nu}^hu - \frac{1}{3}H^{\mu}u_{\mu}\\
			&+ \left( \frac{1}{6}R_g - \frac{1}{2}R_h - \frac{1}{2}|\mathring{L}|_{h}^2 + \frac{1}{3}H^2 \right)\mu(u).
		\end{split}
\end{equation*}
The corresponding quantities on $M$ and $N$ will be denoted by $T_M, \mathcal{L}_M$, etc.

Now define $\theta_0 \in C^{\infty}(\Sigma)$ by allowing $\theta_0$, at each point $p \in \Sigma$, to be the angle at which $M$ and $N$ meet at $p$. That is, $\theta_0$ is defined by $0 < \theta_0 < \pi$ and
$\cos \theta_0 = -\langle \mu_M,\mu_N\rangle_g$. Let $k = g|_{T\Sigma}$, and let $\nu_M, \nu_N$ be the inward-pointing unit normal vectors to $\Sigma$ in $M$ and $N$, respectively. Let $II_M$, $II_N$ be the second fundamental forms
of $\Sigma$ regarded as a submanifold of $M$ or $N$, and computed with respect to $\nu_M,\nu_N$. Let $K$ be the Gaussian curvature of $\Sigma$.
Finally, let $\eta_M, \eta_N$ be the mean curvatures of $\Sigma$ regarded as a hypersurface in $M$ or $N$, respectively.

Define $U \in C^{\infty}(\Sigma)$ by
\begin{equation}
	\label{Ueq}
	U = (\pi - \theta_0)K - \frac{1}{4}\cot(\theta_0)(\eta_M^2 + \eta_N^2) + \frac{1}{2}\csc(\theta_0)\eta_M\eta_N - \frac{1}{3}(\nu_MH_M + \nu_NH_N).
\end{equation}
Now define $G \in C^{\infty}(\Sigma)$ by
\begin{equation}
	\label{Geq}
	G = \frac{1}{2}\cot(\theta_0)(|\mathring{II}_M|_k^2 + |\mathring{II}_N|_k^2) - \csc(\theta_0)\mathring{II}^M_{\alpha\beta}\mathring{II}_N^{\alpha\beta}.
\end{equation}

Finally, define a corner differential operator $P_2^b:C^{\infty}(X) \to C^{\infty}(\Sigma)$ by
\begin{equation}
	\label{P2eq}
\begin{split}
	P_2^bu =& (\theta_0 - \pi)\Delta_ku + \nu_M\mu_Mu + \nu_N\mu_Nu\\
	&+ \cot(\theta_0)(\eta_M\nu_Mu + \eta_N\nu_Nu) - \csc(\theta_0)(\eta_N\nu_Mu + \eta_M\nu_Nu)\\
	&+ \frac{1}{3}(H_M\nu_Mu + H_N\nu_Nu).
\end{split}
\end{equation}

We then have the following results.
\begin{theorem}
	\label{confchangethm}
	Let $(X,g)$ be a compact Riemannian manifold with codimension-two corner $\Sigma^2 = M \cap N$, where $M$ and $N$ are the two embedded hypersurfaces composing $\partial X$. We suppose that $M, N$ meet
	transversely and at positive angle less than $\pi$ (which angle may vary along $\Sigma$).
	Then under the conformal change $\tilde{g} = e^{2\omega}g$, we find
	\begin{align*}
		\widetilde{G} &= e^{-2\omega}G,\numberthis\label{gtranseq}\\
		\widetilde{P}_2^bu &= e^{-2\omega}P_2^bu,\numberthis\label{ptranseq}\\
		\intertext{and}
		e^{2\omega}\widetilde{U} &= U + P_2^b\omega\numberthis\label{vtranseq}
	\end{align*}
\end{theorem}

\begin{theorem}
	\label{gbthm}
	Let $X,M,Q,\Sigma,$ and $g$ satisfy the same hypotheses as Theorem \ref{confchangethm}. Then
	\begin{equation}
		\label{gbeq}
		\begin{split}
			4\pi^2\chi(X^4) =& \int_X\left( \frac{1}{8}|W_g|_g^2 + \frac{1}{2}Q \right)dv_g + \int_M(\mathcal{L}_M + T_M)dv_{h_M}\\
			&+ \int_N(\mathcal{L}_N + T_N)dv_{h_N}  + \oint_{\Sigma}(G + U)dv_k.
		\end{split}
	\end{equation}
\end{theorem}

Both theorems are proved in Section \ref{proofsec}.

In a companion paper in preparation, Matthew Gursky, Aaron Tyrrell, and I will use this to obtain a Gauss-Bonnet formula for asymptotically hyperbolic Einstein manifolds with minimal surfaces meeting the infinite boundary.
This generalizes the result of \cite{and03,cqy08} to the setting of \cite{gw99}.

\section{Background and Proof}
\subsection{The Allendoerfer-Weil Formula}
We begin with the theorem of Allendoerfer-Weil, proved in \cite{aw43}. The proof will simply involve a rewriting of what their formula says, so we take some time to introduce their notation, or a slightly
simplified version of it. The
setting of their theorem is a Riemannian polyhedron, or a compact Riemannian manifold with corners of arbitrary codimension, having non-reentrant angles.

Throughout the paper, the Laplacian $\Delta_g$ will be a negative operator, and the sign of the Riemman tensor will be such that $R_{ij} = g^{kl}R_{iklj}$.

Let $(M^n,g)$ be any Riemannian $n$-manifold, and $M^r$ an embedded $r$-dimensional submanifold, with induced metric $\gamma$.
Near a point in $M^r$, we choose coordinates $\left\{ x^{i} \right\}_{i = 1}^n$
on $M$, such that $\left\{ x^{\mu} \right\}_{\mu = 1}^r$ restrict to coordinates on $M^r$. Let $\sigma$ be a normal vector to $M^r$ in $M^n$.
Following \cite{aw43} (here and throughout this subsection), we define locally along $M^r$ the quantity
\begin{equation*}
	\Lambda_{\mu\nu}(\sigma) = -\sigma_{i}\Gamma_{\mu\nu}^i.
\end{equation*}
Thus, $\Lambda$ is essentially the negative of the second fundamental form in the $\sigma$ direction.

Now for $0 \leq 2p \leq r$ and for $\zeta \in M^r$, define
\begin{equation*}
	\begin{split}
		\Phi_{r,p}(\zeta,\sigma) =& (-1)^p\frac{1}{2^{2p}p!(r - 2p)!}(\det\gamma)^{-1} \varepsilon^{\mu_1\ldots \mu_r}\varepsilon^{\nu_1\ldots\nu_r}R_{\mu_1\mu_2\nu_1\nu_2}\cdots\\
		&\cdot R_{\mu_{2p-1}\mu_{2p}\nu_{2p - 1}\nu_{2p}}
		\Lambda_{\mu_{2p+1}\nu_{2p+1}}(\sigma)\cdots\Lambda_{\mu_r\nu_r}(\sigma).
	\end{split}
\end{equation*}
Here, $R$ is the curvature tensor of $g$ on $M^n$, while $\varepsilon$ is the Levi-Civita symbol. (Note that the factor $(-1)^p$ is not present in the original paper \cite{aw43},
but is here because I am using the opposite sign convention for the curvature tensor.)

For $\zeta \in M^r$, let $S_{\zeta}^{n - r - 1}$ be the unit normal sphere bundle in $N_{\zeta}M^r \subset T_{\zeta}M^n$. Let $\xi \in S_{\zeta}^{n - r - 1}$, and let
$d\xi$ be the area element on $S_{\zeta}^{n - r - 1}$. We define
\begin{equation*}
	\Psi(\zeta|M^r) = \frac{\pi^{-n/2}\Gamma\left( \frac{n}{2} \right)}{2}\sum_{p = 0}^{\lfloor \frac{r}{2}\rfloor}\frac{\Phi_{r,p}(\zeta,\xi)}{(n - 2)(n - 4)\cdots(n - 2p)}d\xi,
\end{equation*}
which is a volume form on $S_{\zeta}^{n - r - 1}$. Now, for $\zeta \in M^r$, let $\Gamma(\zeta) \subset S_{\zeta}^{n - r - 1}$ be the set of unit normal vectors that have
negative inner product with \emph{all} inward-pointing normal vectors (including normal vectors tangent to the boundary). 
This is a spherical cell, bounded by great spheres. We call $\Gamma(\zeta)$ the \emph{outer angle} at $\zeta$.

Meanwhile, on the interior of $M^n$, we define $\Psi \in \Omega^n(M^n)$ by
\begin{equation}
	\label{psieq}
	\Psi(z) = (-2\pi)^{-n/2}\frac{1}{2^n(n/2)!}\det(g)^{-1}\varepsilon^{i_1\ldots i_n}\varepsilon^{j_1\ldots j_n}R_{i_1i_2j_1j_2}\cdots R_{i_{n - 1}i_{n}j_{n - 1}j_n} dv_g,
\end{equation}
if $n$ is even, and by $\Psi = 0$ if $n$ is odd.

Let $P^n$ be a compact Riemannian manifold with corners of arbitrary codimension; we may write $\partial P = \cup_{r = 0}^{n - 1}\cup_{\lambda = 1}^{N_r}P_{\lambda}^r$, where
$P_{\lambda}^r$ is an $r$-dimensional manifold with corners, and if $\lambda \neq \lambda'$, then $P_{\lambda}^r \cap P_{\lambda'}^r$ contains only manifolds (with corner) of dimension less than $r$.

\begin{theorem}[\cite{aw43}, Theorem II]
	With $P^n$ and $\{P_{\lambda}^r\}$ as above, we have
	\begin{equation}
		\label{aweq}
		(-1)^n\chi'(P^n) = \int_{P^n}\Psi + \sum_{r = 0}^{n - 1}\sum_{\lambda}\int_{P_{\lambda}^r}\int_{\Gamma(\zeta)}\Psi(\zeta|P_{\lambda}^r)dv_{\gamma_{\lambda}^r}(\zeta).
	\end{equation}
\end{theorem}
The quantity $\chi'(P^n)$ is the \emph{interior} Euler characteristic, i.e., the Euler characteristic computed using only fully interior cells. For a four-manifold with corners, it coincides with the usual Euler characteristic.
(See \cite{c45} and pp. 154-55 of \cite{l56}. The distinction is important in \cite{aw43} because that paper considers Riemannian polyhedra more general than manifolds with corner.)

\subsection{Proofs of Theorems}
\label{proofsec}
To prove the Gauss-Bonnet formula, we rewrite the terms appearing in (\ref{aweq}) in terms of more geometric quantities. 

\begin{proof}[Proof of Theorem \ref{gbthm}] First, let us set some notation. Recall that we are working on a compact four-manifold
$X^4$ with $\partial X = M^3 \cup N^3$, where $M, N$ are embedded submanifolds with boundary and $\partial M = M \cap N = \partial N = \Sigma^2$, a closed embedded submanifold of (co)dimension two.
We will use indices $1 \leq i, j \leq 4$ on $X$, with $1 \leq \mu,\nu \leq 3$ on $M$ or $N$, and $1 \leq \alpha, \beta \leq 2$ on $\Sigma$.

Let us begin by briefly computing $\Psi(z)$; although the answer is of course the best-known among the quantities above, it will be useful to work through the computation and
to get the constants right. For any metric $g$, a well-known formula tells us that
\begin{equation*}
	(\det g)^{-1}\varepsilon^{i_1\ldots i_n}\varepsilon^{j_1\ldots j_n} = \det\left( g^{i_kj_l} \right)_{k,l=1}^n.
\end{equation*}
Taking $n = 4$, a straightforward calculation shows that
\begin{equation*}
	\varepsilon^{i_1\ldots i_4}\varepsilon^{j_1\ldots j_4}R_{i_1i_2j_1j_2}^gR_{i_3i_4j_3j_4}^g = 4R_g^2 - 16R_g^{ij}R_{ij}^g + 4R^{ijkl}_gR_{ijkl}^g.
\end{equation*}
Now, again for $n = 4$, we define the Schouten tensor $P_{ij}^g$ by
\begin{equation*}
	P_{ij}^g = \frac{1}{2}R_{ij} - \frac{1}{12}R_gg_{ij},
\end{equation*}
and define $J$ by
\begin{equation*}
	J = g^{ij}P^g_{ij}.
\end{equation*}
Recall next that the Riemann curvature tensor of $g$ satisfies
\begin{equation*}
	R_{ijkl}^g = W_{ijkl}^g + P_{il}g_{jk} + P_{jk}g_{il} - P_{ik}g_{jl} - P_{jl}g_{ik},
\end{equation*}
where $W$ is the Weyl tensor. Putting all these and (\ref{psieq}) together, we finally obtain
\begin{equation*}
	\Psi(z) = (2\pi)^{-2}\left[ \frac{1}{8}|W_g|_g^2 - |P_g|_g^2 + J_g^2 \right]dv_g.
\end{equation*}
But $Q = 2(J_g^2 - |P_g|_g^2) - \frac{1}{6}\Delta_gR_g$, so we obtain
\begin{equation}
	\label{xinteq}
	\int_X\Psi(z) = \frac{1}{4\pi^2}\int_X\left( \frac{1}{8}|W_g|_g^2 + \frac{1}{2}Q + \frac{1}{12}\Delta_gR_g  \right) dv_g.
\end{equation}

We next turn to computing $\Psi_{3,0}$ along $M^3 \subset X$. (We choose $M$ just for definiteness -- the formula of course will also hold along $N$.) Let, for now, $h = g|_{TM}$.
For $\sigma \in NM$ and $X,Y \in TM$ all at the same point, let
$L_M(\sigma)(X,Y) = \langle\nabla_X^gY,\sigma\rangle_g$ be the second fundamental form (which of course does not depend on the extension of $Y$ chosen). Thus, $\Lambda_{ij}(\sigma) = -L(\sigma)_{ij}$.
Now, recall that for any $3\times 3$ symmetric matrix $A$,
\begin{equation*}
	\det(A) = \frac{1}{6}\left[ \tr(A)^3 - 3\tr(A)|A|^2 + 2\tr(A^3) \right].
\end{equation*}
(This may be proved, for example, by diagonalizing.) We thus get, along $M$,
\begin{align*}
	\Psi_{3,0}(\zeta,\sigma) &= \frac{1}{6\det(h)}\varepsilon^{i_1i_2i_3}\varepsilon^{j_1j_2j_3}\Lambda_{i_1j_1}(\sigma)\Lambda_{i_2j_2}(\sigma)\Lambda_{i_3j_3}(\sigma)\\
	&= \frac{1}{\det(h)}\det(\Lambda_{ij}(\sigma))\\
	&= -\det(h^{-1}L(\sigma))\\
	&= \frac{1}{6}\left[ -\tr_h(L)^3 + 3\tr_h(L)|L|_h^2 + 2\tr_h(L^3) \right]\\
	&= \frac{1}{6}H|\mathring{L}|_h^2 - \frac{1}{3}\tr_h(\mathring{L}^3) - \frac{1}{27}H^3,
\end{align*}
where $H = h^{\mu\nu}L_{\mu\nu}$, $\mathring{L}$ is the tracefree part of $L$, and the last step follows by expanding the previous line.

Next we consider $\Psi_{3,1}$ along $M^3$. Recalling that $\Lambda = -L$, we apply Gauss's formula and find
\begin{align*}
	\Phi_{3,1}(\zeta,\sigma) &= \frac{1}{4}\left[ h^{\mu_1\nu_1}h^{\mu_2\nu_2}h^{\mu_3\nu_3}-h^{\mu_1\nu_1}h^{\mu_3\nu_2}h^{\mu_2\nu_3} - h^{\mu_1\nu_2}h^{\mu_2\nu_1}h^{\mu_3\nu_3} \right.\\
	&\quad\left. + h^{\mu_1\nu_2}h^{\mu_3\nu_1}h^{\mu_2\nu_3} + h^{\mu_1\nu_3}h^{\mu_2\nu_1}h^{\mu_3\nu_2} - h^{\mu_1\nu_3}h^{\mu_3\nu_1}h^{\mu_2\nu_2}\right]R_{\mu_1\mu_2\nu_1\nu_2}^hL_{\mu_3\nu_3}\\
	&\quad -\frac{1}{4}\varepsilon^{\mu_1\mu_2\mu_3}\varepsilon^{\nu_1\nu_2\nu_3}(L_{\mu_1\nu_2}L_{\mu_2\nu_1}L_{\mu_3\nu_3} - L_{\mu_1\nu_1}L_{\mu_2\nu_2}L_{\mu_3\nu_3})\\
	&=\frac{1}{4}\left[ -HR_h + L^{\mu\nu}R_{\mu\nu}^h - HR_h + 3L^{\mu\nu}R_{\mu\nu}^h \right] + 3\det(L)\\
	&= \mathring{L}^{\mu\nu}R_{\mu\nu}^h - \frac{1}{6}HR_h + \tr_h(\mathring{L}^3) + \frac{1}{9}H^3 - \frac{1}{2}H|\mathring{L}|_h^2.
\end{align*}
(Throughout, $L$ is computed with respect to the unit normal vector $\sigma$ -- of which, for a three-manifold embedded in a four-manifold, there are only two choices.)
Therefore,
\begin{align*}
	\Psi(\zeta|M^3) &= \frac{\pi^{-2}\Gamma(2)}{2}(\Phi_{3,0} + \frac{1}{2}\Phi_{3,1})d\xi\\
	&= \frac{1}{2\pi^2}\left[ \frac{1}{2}\mathring{L}^{\mu\nu}(\xi)R_{\mu\nu}^h - \frac{1}{12}H(\xi)R_h - \frac{1}{12}H(\xi)|\mathring{L}(\xi)|_h^2\right.\\
	&\quad+ \left.\frac{1}{54}H(\xi)^3 + \frac{1}{6}\tr_h(\mathring{L}(\xi)^3) \right]d\xi.
\end{align*}
The measure $d\xi$ here is actually a measure on the $0$-sphere, but we include the $\xi$'s to remind the reader of the dependence of this formula on a choice of normal vector. The only
element $\xi \in \Gamma(\zeta)$ is $-\mu_M$; consequently, we find
\begin{equation}
	\label{minteq}
	\int_M\int_{\Gamma(\zeta)}\Psi(\zeta|M) dv_h(\zeta) = \frac{1}{4\pi^2}\int_M(T_M + \mathcal{L}_M + \frac{1}{3}\Delta_{h_M}H_M + \frac{1}{12}\mu_M(R_g))dv_{h_M}.
\end{equation}
Of course the integral along $N$ is analogous.

We now consider the integrand along $\Sigma^2$. Let $k = g|_{T\Sigma}$. For a normal vector $\sigma \in N\Sigma$, we let $L_{\Sigma}(\sigma)$ be the corresponding second fundamental form. (When
appropriate in index notation, we will write the $\Sigma$ up instead of down.) That is, $L(\sigma)(X,Y) = \langle \nabla_X^gY,\sigma\rangle_g$. We write $II$ for the
\emph{vector} second fundamental form, which is to say, $II(X,Y) = (\nabla_X^gY)^{\perp}$. We also write $II_M(X,Y) = L(\nu_M)(X,Y)$ and $II_N(X,Y) = L(\nu_N)(X,Y)$. 

First, we find
\begin{align*}
	\Phi_{2,0}(\zeta,\sigma) &= \frac{1}{2\det(k)}\varepsilon^{\alpha_1\alpha_2}\varepsilon^{\beta_1\beta_2}\Lambda_{\alpha_1\beta_1}(\sigma)\Lambda_{\alpha_2\beta_2}(\sigma)\\
	&= \det(k^{-1}L_{\Sigma}(\sigma))\\
	&= \frac{1}{2}(H_{\Sigma}(\sigma)^2 - |L_{\Sigma}(\sigma)|_k^2).
\end{align*}
Here $H_{\Sigma}(\sigma) = k^{\alpha\beta}L_{\Sigma}(\sigma)_{\alpha\beta}$.

Next, we find
\begin{align*}
	\Phi_{2,1}(\zeta,\sigma) &= -\frac{1}{4\det(k)}\varepsilon^{\alpha_1\alpha_2}\varepsilon^{\beta_1\beta_2}R_{\alpha_1\alpha_2\beta_1\beta_2}^g\\
	&= -\frac{1}{\det(k)}R_{1212}^g\\
\end{align*}
By Gauss's equation, we have
\begin{align*}
	R_{1212}^g &= R_{1212}^k - \langle II(\partial_1,\partial_2),II(\partial_1,\partial_2)\rangle + \langle II(\partial_1,\partial_1),II(\partial_2,\partial_2)\rangle\\
	&= R_{1212} - II_{12i}II_{12}^i + II_{11i}II_{22}^i.
\end{align*}
Now let $\mu_M$ be the unit normal to $M$ in $X$, and let $\nu_M$ be the unit normal to $\Sigma$ in $M$. Together, these vectors form an orthonormal basis at each point for $N\Sigma$. Thus, we find
\begin{align*}
	\Phi_{2,1}(\zeta,\sigma) &= K - \det(L_{\Sigma}(\nu_M)) - \det(L_{\Sigma}(\mu_M))\\
	&= K - \frac{1}{2}(H_{\Sigma}(\nu_M)^2 + H_{\Sigma}(\mu_M)^2 - |L_{\Sigma}(\nu_M)|_k^2 - |L_{\Sigma}(\mu_M)|_k^2.
\end{align*}
(Recall that $K$ is the Gaussian curvature of $\Sigma$.)
Thus,
\begin{align*}
	\Psi(\zeta|\Sigma^2) &= \frac{1}{2\pi^2}\left[ \Phi_{2,0} + \frac{1}{2}\Phi_{2,1} \right]\\
	&= \frac{1}{2\pi^2}\left[ \frac{1}{2}H_{\Sigma}(\xi)^2 - \frac{1}{2}|L_{\Sigma}(\xi)|_{k}^2 + \frac{1}{2}K - \frac{1}{4}H_{\Sigma}(\nu_M)^2 - \frac{1}{4}H_{\Sigma}(\mu_M)^2\right.\\
	&\quad\left. + \frac{1}{4}|L_{\Sigma}(\nu_M)|_k^2 + \frac{1}{4}|L_{\Sigma}(\mu_M)|_k^2\right]d\xi.
\end{align*}
Here, $d\xi$ is the measure on the circle $S^1$, viewed as the fiber of the unit normal bundle $S^1N\Sigma$ to $\Sigma$ at $\zeta$.

Now, any $\xi \in S^1N\Sigma$ can be written $\xi = \cos(\theta)\nu_M + \sin(\theta)\mu_M$. The outer angle $\Gamma(\zeta)$ can then be written
\begin{equation*}
	\Gamma(\zeta) = \left\{ \cos(\theta)\nu_M + \sin(\theta)\mu_M: \theta_0(\zeta) + \frac{\pi}{2} < \theta < \frac{3\pi}{2} \right\}.
\end{equation*}

The contribution to the Gauss-Bonnet formula from $\Sigma_2$ will therefore be
\begin{equation*}
\begin{split}
	I_{\Sigma} :=& \frac{1}{4\pi^2}\oint_{\Sigma}\int_{\theta_0(\zeta) + \frac{\pi}{2}}^{\frac{3\pi}{2}}\left( K - \frac{1}{2}H_{\Sigma}(\nu_M)^2 + \frac{1}{2}H_{\Sigma}(\mu_M)^2 + \frac{1}{2}|L_{\Sigma}(\nu_M)|_k^2\right.\\
		&+\frac{1}{2}|L_{\Sigma}(\mu_M)|_k^2
	+ H_{\Sigma}(\cos(\theta)\nu_M + \sin(\theta)\mu_M)^2\\
	&\left.- |L_{\Sigma}(\cos(\theta)\nu_M+\sin(\theta)\mu_M)|_k^2\right)d\theta dv_k(\zeta)
\end{split}
\end{equation*}
Observe that, by bilinearity of $g$
\begin{align*}
	H_{\Sigma}(\cos(\theta)\nu_M + \sin(\theta)\mu_M)^2 &= \cos^2(\theta)H_{\Sigma}(\nu_M) + 2\sin(\theta)\cos(\theta)H_{\Sigma}(\nu_M)H_{\Sigma}(\mu_M)\\
	&\quad+ \sin^2(\theta)H_{\Sigma}(\mu_M),\\
	\intertext{and}
	|L_{\Sigma}(\cos(\theta)\nu_M + \sin(\theta)\mu_M)|_k^2 &= \cos^2(\theta)|L_{\Sigma}(\nu_M)|_k^2\\
	&\quad+ 2\sin(\theta)\cos(\theta)L_{\Sigma}(\nu_M)_{\alpha\beta}L_{\Sigma}(\mu_M)^{\alpha\beta}\\
	&\quad +\sin^2(\theta)|L_{\Sigma}(\mu_M)|_k^2.
\end{align*}
Now,
$\int_{\theta_0+\frac{\pi}{2}}^{\frac{3\pi}{2}}\cos^2(\theta)d\theta = \frac{1}{2}(\pi - \theta_0) + \frac{\sin(2\theta_0)}{4}$,
$\int_{\theta_0+\frac{\pi}{2}}\sin^2(\theta)d\theta =\frac{1}{2}(\pi-\theta_0)-\frac{\sin(2\theta_0)}{4},\text{ and}$, and
$\int_{\theta_0+\frac{\pi}{2}}^{\frac{3\pi}{2}}\sin(\theta)\cos(\theta)d\theta = \frac{1}{4}-\frac{1}{4}\cos(2\theta_0)$.
Also,
\begin{equation*}
	L_{\Sigma}(\xi)_{\alpha\beta}L_{\Sigma}(\tau)^{\alpha\beta} = \mathring{L}_{\Sigma}(\xi)_{\alpha\beta}\mathring{L}_{\Sigma}(\tau)^{\alpha\beta} + \frac{1}{2}H_{\Sigma}(\xi)H_{\Sigma}(\tau).
\end{equation*}
Next, we can write $\nu_N = \cos(\theta_0)\nu_M + \sin(\theta_0)\mu_M$, so 
\begin{equation}
	\label{muMeq}
	\mu_M = \csc(\theta_0)\nu_N - \cot(\theta_0)\nu_M.
\end{equation}
Finally, we have $\mathring{L}_{\Sigma}(\nu_M) = \mathring{II}_M$, and $\mathring{L}_{\Sigma}(\nu_N) = \mathring{II}_N$, while
$H_{\Sigma}(\nu_M) = \eta_M$ and $H_{\Sigma}(\nu_N) = \eta_N$.
Applying all these identities, we get
\begin{align}
	\label{siginteq}
	I_{\Sigma} &= \frac{1}{4\pi^2}\oint_{\Sigma}\left((\pi - \theta_0)K - \frac{1}{4}\cot(\theta_0)(\eta_M^2 + \eta_N^2)\right.\\
	&\quad+ \left.\frac{1}{2}\csc(\theta_0)\eta_M\eta_N + \frac{1}{2}\cot(\theta_0)(|\mathring{II}_M|_k^2 + |\mathring{II}_N|_k^2) - \csc(\theta_0)\mathring{II}_{\alpha\beta}^M\mathring{II}^{\alpha\beta}_N\right)dv_k\\
	&=\frac{1}{4\pi^2}\oint_{\Sigma}\left(U + G + \frac{1}{3}\left(\nu_MH_M + \nu_NH_N\right)\right)dv_k,
\end{align}
where $U, G$ are as in (\ref{Ueq}),(\ref{Geq}).

Now applying (\ref{xinteq}), (\ref{minteq}), (\ref{siginteq}), and Green's theorem yields (\ref{gbeq}), and thus the theorem.
\end{proof}

One observes that, in order to write the Chern-Gauss-Bonnet formula in terms of quantities $Q$ and $T$ that would satisfy (\ref{QPeq}), (\ref{TPeq}), it was necessary for Chang and Qing to
add divergence terms both to the interior and boundary integrals (as well as a correction term at the boundary, via Green's theorem). On the other hand, as Theorem \ref{confchangethm} and the above calculation shows, 
there is no need to add any divergence to the corner integral (though we still need the Green term); this is fortuitous, as there is no obvious second-order divergence to add!

\begin{proof}[Proof of Theorem \ref{confchangethm}]
	Suppose that $\tilde{g} = e^{2\omega}g$, where $\omega \in C^{\infty}(X)$. Recall the following standard formulas along $\Sigma$:
	\begin{align*}
		\widetilde{K} &= e^{-2\omega}(K - \Delta_k\omega)\\
		\widetilde{\eta}_M &= e^{-\omega}(\eta_M - 2\nu_M\omega)\\
		\widetilde{\eta}_N &= e^{-\omega}(\eta_N - 2\nu_N\omega)\\
		\mathring{\widetilde{II}}_M &= e^{\omega}\mathring{II}_{M}\\
		\mathring{\widetilde{II}}_N &= e^{\omega}\mathring{II}_N\\
		\widetilde{H}_M &= e^{-\omega}(H - 3\mu_M\omega)\\
		\widetilde{H}_N &= e^{-\omega}(H - 3\mu_N\omega)\\
		\tilde{\theta}_0 &= \theta_0.
	\end{align*}
	Using these and equation (\ref{Geq}), (\ref{gtranseq}) follows easily.

	We show (\ref{vtranseq}) explicitly. Using (\ref{Ueq}), we find
	\begin{align*}
		\widetilde{U} &= (\pi - \theta_0)e^{-2\omega}(K_{\Sigma} - \Delta_k\omega) - \frac{1}{4}\cot(\theta_0)e^{-2\omega}\left( \eta_M^2 - 4\eta_M\nu_M\omega + 4\eta_M(\omega)^2\right.\\
		&\quad+ \left.\eta_N^2 - 4\eta_N\nu_N\omega + 4\eta_N(\omega)^2\right)\\
		&\quad + \frac{1}{2}\csc(\theta_0)e^{-2\omega}\left( \eta_M\eta_N - 2\eta_M\nu_N\omega - 2\eta_N\nu_M\omega + 4\eta_M(\omega)\eta_N(\omega) \right)\\
		&\quad- \frac{1}{3}e^{-\omega}\nu_M(e^{-\omega}(H_M - 3\mu_M\omega))
		- \frac{1}{3}e^{-\omega}\nu_N(e^{-\omega}(H_N - 3\mu_N\omega));\\
		\intertext{Now applying (\ref{muMeq}) to the last two terms, we find}
		\widetilde{U} &= e^{-2\omega}\left[(\pi - \theta_0)K_{\Sigma} - \cot(\theta_0)(\eta_M^2 + \eta_N^2) + 2\csc(\theta_0)\eta_M\eta_N - \frac{1}{3}(\nu_MH_M + \nu_NH_N)\right.\\
		&\quad+ \cot(\theta_0)\eta_M\nu_M\omega - \cot(\theta_0)\nu_M(\omega)^2 + \cot(\theta_0)\eta_N\nu_N\omega - \cot(\theta_0)\nu_N(\omega)^2\\
		&\quad- \csc(\theta_0)\eta_M\nu_N\omega
		- \csc(\theta_0)\eta_N\nu_M\omega + 2\csc(\theta_0)\nu_M(\omega)\nu_N(\omega) + (\theta_0 - \pi)\Delta_k\omega\\
		&\quad+ \frac{1}{3}H_M\nu_M\omega + \nu_M\mu_M\omega - \csc(\theta_0)\nu_M(\omega)\nu_N(\omega) + \cot(\theta_0)\nu_M(\omega)^2\\
		&\quad + \frac{1}{3}H_N\nu_N\omega + \nu_N\mu_N\omega - \csc(\theta_0)\nu_N(\omega)\nu_M(\omega) + \cot(\theta_0)\nu_N(\omega)^2\\
		&= e^{-2\omega}\left( U + P_2^b\omega \right).
	\end{align*}
	It is interesting in the above to note that nonlinear terms coming from the ``Green terms'' $\nu_MH_M + \nu_NH_N$ in $U$ precisely cancel those arising from the Allendoerfer-Weil formula.

	We now need show only (\ref{ptranseq}). This will follow from (\ref{vtranseq}) and the linearity of $P_2^b$, as in proposition 3.1 of \cite{cq97i}. Let $\phi \in C^{\infty}(X)$.
	Suppose $\hat{g} = e^{2(\omega + \phi)}g = e^{2\phi}\tilde{g}$. Then
	\begin{align*}
		P_2^b(\omega + \phi) + U &= e^{2(\omega + \phi)}\widehat{U}\\
		\widetilde{P}_{2}^b\phi + \widetilde{U} &= \widehat{U}e^{2\phi}
	\end{align*}
	Multiply the second equation by $e^{2\omega}$ and subtract to get
	\begin{equation*}
		P_2^b(\omega + \phi) + U - e^{2\omega}\widetilde{P}_2^b\phi - e^{2\omega}\widetilde{U} = 0.
	\end{equation*}
	The result now follows by (\ref{vtranseq}) and linearity.
\end{proof}

\textbf{Acknowledgments} I am indebted to Alice Chang, Paul Yang, Matthew Gursky, and Baris Coskunuzer for helpful conversations. This work was carried out at Princeton University and the University of Texas at Dallas, and I thank both
institutions for their support and for the excellent environments for doing math that they have provided. The work at Princeton was supported also by NSF RTG DMS-1502525.

\bibliographystyle{alpha}
\bibliography{norm}

\end{document}